\documentclass[a4paper,12pt]{article}
\usepackage{amsmath}
\usepackage{amscd}
\usepackage{amsthm}
\usepackage{amssymb} \usepackage{latexsym}
\usepackage{eufrak}
\usepackage{euscript}
\usepackage{epsfig}
\usepackage{tikz}
\usepackage{graphics}
\usepackage{array}
\usepackage{enumerate}
\usepackage{authblk}

\theoremstyle{plain}
\newtheorem{thm}{Theorem}[section]
\newtheorem{lem}[thm]{Lemma}

\theoremstyle{definition}

\theoremstyle{remark}

\theoremstyle{corollary}
\newtheorem{cor}{Corollary}[section]

\newcommand{\bel}[1]{\begin{equation}\label{#1}}

\newcommand{\be}{\begin{equation}}
\newcommand{\ba}{\begin{eqnarray}}
\newcommand{\ea}{\end{eqnarray}}

\newcommand{\qe}{\end{equation}}

\begin{document}
\setcounter{Maxaffil}{2}
\title{On Extension of Regular Graphs}
\author[1,2]{\rm Anirban Banerjee}
\author[1]{\rm Saptarshi Bej}
\affil[1]{Department of Mathematics and Statistics}
\affil[2]{Department of Biological Sciences}
\affil[ ]{Indian Institute of Science Education and Research Kolkata}
\affil[ ]{Mohanpur-741246, India}
\affil[ ]{\textit {\{anirban.banerjee, sayan7299\}@iiserkol.ac.in}}

\maketitle

\begin{abstract}
In this article, we  discuss when one can extend an $r$-regular graph to an $r+1$ regular by  adding edges.  Different conditions on the number of vertices $n$ and regularity $r$ are developed.
% for extending the regularity.
We  derive an upper bound of $r$, depending on $n$, for which, every regular graph $G(n,r)$ can be extended to an $r+1$-regular graph with $n$ vertices.
 Presence of  induced complete bipartite subgraph and complete subgraph is discussed, separately, for the extension of regularity.
\end{abstract} 

\textbf{Keywards:} Regular graph; Matching on regular graph; Bipartite regular graph; Extension of regularity.

%%%%%%%%%%%%
%%%%%%%%%%%%
\section{Introduction}
 A graph $G(n,r)$ with $n$ vertices is called regular of degree $r$ if each of its vertices has the degree $r$. 
 Many studies have been done on the regular graphs till date. Not only, these graphs are interesting in network theory \cite{Bullmore EtAl}, but also, they are quite fascinating geometrically \cite{Pach}.
Several graphs like Moore graph, Cage graph, Petersen graph, etc.~use the concept of regular graphs \cite{Biggs}. 
Coloring on regular graphs is also well studied \cite{Galvin}. Presence of Hamiltonian cycle in random regular graphs has been explored by Fenner and Frieze \cite{Fenner EtAl}.
B\'ela Bollob\'as has  extensively studied several properties of random regular graphs \cite{Bollobas}. 
Matchings in random regular graphs has  been explored in \cite{Petros}. 
Hall has provided matchings on bipartite regular graphs \cite{Bondy EtAl}. A general condition on the existence of matching in any graph has been introduced by Tutte \cite{Bondy EtAl}. Petersen has used it to find some interesting results on 3-regular graphs \cite{Bondy EtAl}. 
Spectral property of regular graphs is a great interest for many researchers \cite{Cvetkovic EtAl,Koledin EtAl}. 
The eigenvalues of the adjacency matrix have been used to study matching on regular graphs \cite{Koledin EtAl,Sebastian EtAl}.

The exact enumeration of non-isomorphic regular graphs with a given number of vertices and regularity is combinatorially hard to figure out.  It could be estimated for some lower number of vertices and regularities. The generalization or any closed formula have not been established. One may observe that there is a certain relationship between the number of non-isomorphic regular graphs with  a given number of vertices $n$ and  regularity $r$. It can  also be seen that there is an order relationship between the numbers of non-isomorphic regular graphs $G(n, r)$ and $G(n, r+1)$, respectively. If we increase $r$ from $0$ to $n-1$, the number of regular graphs up to isomorphism increases until $r$ reaches, approximately, the value $n/2$, then the number of regular graphs  gradually decreases. 
%as $r$ increases approximately until $r$ reaches the value $n/2$. If we increase $r$ further, the  number of non-isomorphic regular graphs decreases gradually. 
Now, one may think if there is any recursive relation between the number of non-isomorphic regular graphs and a  given number of vertices, and regularity. This motivates us to study when a regular graph can be extended to another regular graph of higher regularity by adding edges. In this article, we explore the situations when one can or can not extend the graph $G(n,r)$ to $G(n,r+1)$ by adding $n/2$ edges.
\begin{lem}[Dirac's theorem]
\label{intro:lem1}
 If a connected graph $G$ has $n \geq 3$ vertices and the degree of each vertex is at least $n/2$, then $G$ is Hamiltonian.
\end{lem}

\begin{thm}
If $G(n,r)$ is a regular graph with $r < n/2$ and $n$ is even, then  $G(n,r)$ can always be extended to $G(n,r+1)$.
\end{thm}
\label{intro:thm1}
\begin{proof}
%  Let $G(n,r)$ be a regular graph with $r < n/2$ and $n$ be even. 
 $G^{c}$ is a regular graph where the degree of each vertex is $(n-1-r)$. 
Since  $r <n/2 $, 
$ (n-1-r) \geq n/2$.
 Thus, there exists a Hamiltonian even-cycle in $G^{c}$ (by lemma (\ref{intro:lem1})). 
 %Let us index the vertices of this Hamiltonian cycle as $v_{1},v_{2},v_{3},....,v_{n},v_{1}$ where $n$ is an even number. Now,  the  $n/2$ pair of vertices, $v_{1}v_{2},v_{3}v_{4},...,v_{n-1}v_{n}$, are not connected in $G$. Connecting these pair of vertices by edges $G(n,r)$ can be extended to $G(n,r+1)$.
 Hence the proof.
\end{proof}

As we  see, when $n$ is even, for any $r < n/2$  $G(n,r)$ can be extended to $G(n,r+1)$, which is  not  always true for $r \ge n/2$. 
Hence, $n/2$ can be considered as an upper bound for $r$. In the next section, we attempt to find a better upper bound (say $ \mathbb{K} > n/2$) for $r$, such that, all $G(n, r)$ where $n$ is even and $r \le \mathbb{K} $ can be extended to $G(n,r+1)$. For $n, r$ are even,  we estimate the value of $\mathbb{K} $. 

%%%%%%%%%%%%
%%%%%%%%%%%%
\section{Maximum value of $r$ for  the extension of  $G(n, r)$}

\begin{lem}
\label{tutte}
(Tutte's Theorem). A graph $G$ has a perfect matching if and only if $odd(G-S) \leq |S|$ for all $S \subset V$, where $odd(G-S)$ represents the number of odd components in $G-S$. \cite{Bondy EtAl}
\end{lem}
A \textit{balloon}, in a graph, is a maximal $2$-edge-connected subgraph incident to exactly one cut edge of that graph  \cite{Suil EtAl}.
Note that, in an $r$-regular graph, where $r$ is odd, the minimum number of vertices in a balloon is $r+2$ (see construction($2.1$) in \cite{Suil EtAl}). Let us denote the number of balloons in a graph $G$ by $b(G)$.

\begin{lem}\label{Ch4:L0}
 \cite{Suil EtAl} Let $G(n, r)$ be a regular graph, and let $S$ be a subset of $V(G)$. If the number of edges from each odd component of $G-S$ to $S$ is only $1$ or at least $r$ then $odd(G-S)-|S| \leq \frac{r}{r-1}b(G)$, where $odd(G-S)$ is the number of connected components in $G-S$ having odd number of vertices.
\end{lem}

\begin{lem}
\label{Ch4:L1}
For any real value $r>15$ and $k$,
$(k+2)r-k^{2}+2 > 3r+7,$
if $2 \le k \le r-2$.
\end{lem}
\begin{proof}
For any  value of $r > 15$, the equation $f(k) = (k+2)r-k^{2}+2-(3r+7)$ is a parabola. Now, $f(k) > 0$ between  $\frac{r- \sqrt {r^{2}-4(r+5)}}{2}$ and $\frac{r+ \sqrt {r^{2}-4(r+5)}}{2}$, which are two distinct real roots of the equation $f(k)=0$.
Since $r>15$, 
$\frac{r- \sqrt {r^{2}-4(r+5)}}{2} < 2 \le k \le r-2 <  \frac{r+ \sqrt {r^{2}-4(r+5)}}{2}.$
Thus, for $r>15$ and $2 \le k \le r-2$, $(k+2)r-k^{2}+2 > 3r+7$.
\end{proof}

\begin{lem}
\label{ineq2}
If for any real vale $x$ and $r$, $1 \leq x \leq r$, then $x(r-x+1) \geq r$.
\end{lem}
\begin{proof}
Let $f(x)=x(r-x+1)-r$. 
Since both the roots of  $f(x)$, which are $1$ and $r$, are real and the coefficient of $x$ in $f(x)$ is negative, thus 
$f(x) \geq 0$ when $1 \leq x \leq r$. 
\end{proof}

\begin{lem}
\label{Ch4:L2}
Any regular graph $G(n,r)$  has a perfect matching if $r>15$ is odd, $n$ is even and $n<3r+7$.
\end{lem}
\begin{proof}
We prove this result by contradiction. 
Let us assume that $G(n,r)$ has no perfect matching. Now, by Tutte's theorem, there exists a set $S \subset V(G)$ such that $|S|<odd(G-S)$. 
Note that, $|S|\neq 0$, since, if $|S|=0$, then, there  exists at least one odd component (a component having odd number of vertices) in $G$, but, as $r$ is odd, such a component can not exist in $G$. Hence, $|S| \geq 1$.
Let $|S|=k$. Now, if $k$ is odd, then $odd(G-S)$ is odd, and if $k$ is even, $odd(G-S)$ must be even. Since, $G(n,r)$ has no perfect matching,  $odd(G-S) \geq k+2$, i.e., $odd(G-S) \geq 3$.\\
%because if $odd(G-S)$ were even then number of vertices in all the odd components together would be even. Since we know that number of vertices in all even components together is always even and $|S|$ is odd, then the sum of vertices of odd components, even components, and $S$ would yield an odd number. But, the sum of vertices of odd components, even components, and $S$ is equal to $V(G)$. Hence, if $k$ is odd then $odd(G-S)$ must be odd. Similarly, if $k$ is even, $odd(G-S)$ must be even.\\

\textbf{Case 1: When there is at least one isolated vertex  in $G-S$.}\\
Let the number of isolated vertices in $G-S$ be $x$.
Since, $x>0$, it is evident that 
\bel{Ch4:L2:e2.9}
|S|=k \geq r. 
\qe
Now, 
\bel{Ch4:L2:e3}
odd(G-S) \geq k+2 \geq r+2.
\qe
Since, $n \le 3r+5$, there must be at least $2r+5$ vertices in $S^{c}$. 
Hence, the number of odd components, in $G-S$, having more than one vertex is less than $(2r+5-x)/3$.
Thus, $odd(G-S) < (2r+5+2x)/3$. Using equation (\ref{Ch4:L2:e3}), we get 
$(2r+5+2x)/3 >r+2$, i.e., 
\bel{Ch4:L2:e4}
x> r/2.
\qe

\textbf{Case 1.1: All odd components are having the number of vertices $\leq r$.}\\
The minimum size of a balloon in an r-regular graph, where r is odd,  is $r+2$. Hence, $b(G) \le 2$ (since, to construct 3 balloons, the minimum number of vertices required is $3r+6$, but $n \leq 3r+5$). 
 Now, if $y$ be the number of vertices in an odd component, then, the minimum number of edges from that component to $S$ must be at least $y(r-y+1)\geq r$  (by lemma (\ref{ineq2})). 
Since, $b(G) \leq 2$, using lemma (\ref{Ch4:L0}) we get, $odd(G-S)-|S| \leq \frac{r-1}{r}b(G)<2$, i.e., 
 $odd(G-S)<|S|+2.$ This contradicts the inequality (\ref{Ch4:L2:e3}). Hence, the case (1.1) is impossible to arise.\\

\textbf{Case 1.2: There is an odd component that has more than $r$ vertices.}\\
Note that, there can be exactly one odd component which has more than $r$ vertices, i.e.,  at least $r+2$ vertices.
 If there are such two, then, the minimum number of vertices in $G$ becomes
 $2(r+2) + r/2 + r$ (by equation (\ref{Ch4:L2:e4}) and (\ref{Ch4:L2:e2.9})), which is  $> 3r+5$, and this is not possible, since $n \le 3r+5$.
Now, if there is an odd component with at least $r+2$ vertices, $odd(G-S)$ must be $r+2$ and $|S|=r$ (since $odd(G-S)\geq |S|+2$ and $n \le 3r+5$).  This implies $x = r$, since, $x$ can not be greater than $|S|$.
Now, since, $|S|=r$ and $x = r$, there must exist another odd component with 3 vertices and which can not be connected to $S$ in $G$, but since, $G$ is an $r$(odd)-regular graph, $G$ can not have any isolated odd component. This leads a contradiction. Thus,  the case (1.2) is  also impossible to arise. Hence, there can not be any isolated vertex in $G-S$.\\

\textbf{Case 2: When there is no isolated vertex  in $G-S$.}\\
Here, the minimum possible degree of a vertex in an odd component of $G-S$ must be $r-k$. Thus, the minimum number of vertices in that component becomes $(r-k)+1$. Since, $odd(G-S) \geq k+2$, the minimum number of vertices in $G$ must be $k+(k+2)(r-k+1)$. Hence, 
\bel{Ch4:L2:e1}
3r+7 > n \ge (k+2)r-k^{2}+2.
\qe
Now, $odd(G-S) \le (3r+7-k)/3$, since here, the minimum number of vertices in an odd component is $3$. Thus, $k<(3r+7-k)/3$, since, $|S|<odd(G-S)$. This implies  $k<(3r+7)/4$. 
Note that, for  $r>15$, $(3r+7)/4 < r-2$. Thus, 
\bel{Ch4:L2:e2}
1 \le k<r-2,\text{ when }r>15.
\qe
Now the equation (\ref{Ch4:L2:e1}), (\ref{Ch4:L2:e2}) and lemma (\ref{Ch4:L1}) (which shows that if $r>15$ and  $2 \leq k \leq r-2$, then $(k+2)r-k^{2}+2 > 3r+7$) claim that 
$k=1.$
This implies that the minimum possible degree of a vertex  and the minimum possible number of vertices in an odd component are $r-1$ and $r$, respectively. 
Now, there must be a vertex in the odd component with the degree $r$, otherwise, if all the vertices in the odd component have degree $r-1$, then the odd component must have $r$ number of vertices and all of them must be connected to the single vertex in $S$ in $G$. Hence, $odd(G-S)=1$ which is a contradiction, since,  $odd(G-S) \geq 3$.
So, there must be at least one vertex, in an odd component, of  degree $r$. Thus, there are at least $r+2$ vertices in an odd component.
Now, since, $odd(G-S) \geq 3$, the minimum number of vertices present in $G$ is $3(r+2)+1 = 3r+7$, which is a contradiction, because $n<3n+7$. 

Hence, our assumption, $G(n,r)$ has no perfect matching, is wrong.
\end{proof}

\begin{cor}
\label{Ch4:C1}
Let $G(n,r)$ be a regular graph such that $n$ is even, $r$ $(>15)$ is odd, and $r \geq n/4$. If $G(n,r)$ is not connected, then $G(n,r)$ has a complete matching.
\end{cor}
\begin{proof}
%Let $G(n,r)$ be a disconnected regular graph satisfying the hypothesis of the corollary. 
Since $G$ is $r$-regular, each component of $G$ has at least $r+1$, but, less than $3r+7$ vertices.
Thus, there exists a perfect matching in every component of $G$ (using lemma(\ref{Ch4:L2})). Hence, $G(n,r)$ has a perfect matching.
\end{proof}

\begin{thm}
Let $n$  and $r$ be even. If $r<2(n+2)/3$ when $n \geq 52$, and $r<n-16$ when $n<52$,
then $G(n,r)$ can always be extended to $G(n,r+1)$.
\end{thm}
\begin{proof}
Since $n$  and $r$ are even, $G^c$ is an $r'$-odd regular graph of $n$ vertices, where $r'=n-r-1$. Now, it is sufficient to prove $G^c$ has a perfect matching.\\

\textbf{Case 1:} When $r<2(n+2)/3$ and $n \geq 52,$
\bel{Ch4:Th1:e1}
r<2(n+2)/3 \Rightarrow n< (3r' +7),
\qe
and 
\bel{Ch4:Th1:e2}
n \geq 52 \text{ and } (\ref{Ch4:Th1:e1}) \Rightarrow r'>15.
\qe
Thus, by lemma(\ref{Ch4:L2}), $G^c$ has a perfect matching.\\

\textbf{Case 2:} When $r<n-16$ and $n<52,$
\bel{Ch4:Th1:e3}
r<n-16 \Rightarrow r'>15,
\qe
and 
\bel{Ch4:Th1:e4}
 n< 52  \text{ and } (\ref{Ch4:Th1:e3}) \Rightarrow (3r' +7)>n.
\qe
Thus, by lemma(\ref{Ch4:L2}), $G^c$ has a perfect matching.
\end{proof}

\begin{thm}
\label{Ch4:Th2}
Let $G(n,r)$ be a regular graph, where $n$ and $r$ are even, such that, either
$r<3n/4$ when $n\geq 64$
 or
  $r<n-16$ when $n<64$.
If there exists a complete bipartite subgraph $H$ of $G$, such that, $V(G)=V(H)$,  then $G(n,r)$ can be extended to $G(n,r+1)$.
\end{thm}
\begin{proof}
 $G^{c}(n,r')$ is  a regular graph, such that, $n$ is even, $r'>15$ is odd, and $r' \geq n/4$. Note that, $G^{c}$ is disconnected. Hence, by the corollary(\ref{Ch4:C1}), $G^{c}$ has a perfect matching. Thus, $G(n,r)$ can be extended to $G(n,r+1)$.
\end{proof}

%%%%%%%%%%%%
%%%%%%%%%%%%

\section{Regular graphs containing $K_{n_1,n_2}$ or $K_{n/2}$.}

\begin{thm}
If $n$ is even, $r \geq n/2$, and  $G(n,r)$ contains an induced complete bipartite subgraph with the same $n$ vertices, such that, the cardinality of each partition is odd, then it is impossible to extend $G(n,r)$ to $G(n,r+1)$.
\end{thm} 
\begin{proof}  
Since, $G(n,r)$ contains an induced complete bipartite subgraph with the same $n$ vertices, such that, the cardinality of each partition is odd, then
$G^c$  contains two disjoint components, each of them has odd number of vertices. $G^c$ has no perfect matching (by lemma(\ref{tutte})).  
Hence the proof.
\end{proof}

Let $G(V,E)$ be a graph with the vertex set $V(G)$. Let $S$ be a subset of $V(G)$ and $N_{G}(S)$ be the set $\{v : v$ is a neighbour of $w \in S\}$. Now, to prove the next theorem we use Hall's theorem as a lemma.
\begin{lem}
\label{hall:lem1}
  (Hall's theorem). Suppose, $G$ is a bipartite graph with the bi-partitions $A$ and $B$ of the vertex set of $G$.
Now, $G$ has a perfect matching if and only if $|A| = |B|$ and for any subset $S$ of $A$, $|N_{G}(S)| \geq |S|$. \cite{Bondy EtAl}
\end{lem}
A corollary of Hall's theorem is as follows.
\begin{lem}
\label{hall:lem2}
Any regular bipartite graph $G(n,r)$, where $n$ is even, with $n/2$ vertices in each partition has a complete matching.
\end{lem}

\begin{lem}
\label{hall:lem3}
If $G(n,r)$ be a regular graph, where $n/2$ is even, $r \geq n/2$, and $K_{n/2}$ is a subgraph of $G(n,r)$, then $G^{c}$ is bipartite.
\end{lem}
\begin{proof}   
%The number of edges in $G(n,r)$ and in $K_{n/2}$ are $nr/2$ and  $(n/4)(n/2-1)$, respectively. 
%
%Let us consider that the vertices, $v_{1}, v_{2}, ..., v_{n/2}$ in $G$ form a subgraph $K_{n/2}$. %Now, each  $v_{i}, 1 \leq i\leq n/2$, has exactly $r-(n/2-1)$ neighbours in $V(G(n,r))-\{v_{1},v_{2},v_{3}, ..., v_{n/2}\}$. 
The total number of edges attached to the set of vertices $\{v_{1},v_{2},...,v_{n/2}\}$, which form a subgraph $K_{n/2}$, is $(n/4)(n/2-1)+n/2[r-(n/2-1)]$. The number of the remaining edges is 
%$nr/2-[(n/4)(n/2-1)+n/2[r-(n/2-1)]=$
$(n/4)(n/2-1)$. 
Hence, all other vertices, $v_{(n/2)+1}, v_{(n/2)+2}, ..., v_{n}$, are  connected to each other and form another subgraph $K_{n/2}$. 
Thus, $G^{c}$ is bipartite with the bipartition: $\{v_{1},v_{2},v_{3}, ..., v_{n/2}\}$ and $\{v_{(n/2)+1},v_{(n/2)+2},v_{(n/2)+3}, ..., v_{n}\}$.
\end{proof}

\begin{thm}
If $G(n,r)$ be a regular graph, where $n$ is even, $r \geq n/2$, and $K_{n/2}$ is a subgraph of $G(n,r)$, then $G(n,r)$ can be extended to $G(n,r')$ for any $r' \leq n-1$.
\end{thm}
\begin{proof} $G^{c}$ is a bipartite (with  $n/2$ vertices in each partition) and regular graph (by lemma (\ref{hall:lem3})). Thus, $G^{c}$  has a complete matching (by lemma (\ref{hall:lem2})), say $M_{1}$. 
Clearly, $[G^{c}-M_{1}]^{c}$ is a regular graph with $n$ vertices, $r+1$ regularity, and has a subgraph $K_{n/2}$. Hence, $G^{c}-M_{1}$ also has a complete matching, say $M_{2}$. 
Similarly, $G^{c}-M_{1}- M_{2}$, which is a  regular bipartite graph,  has a complete matching. 
We can repeat 
this process until we delete $n-1-r$ complete matchings. This implies, it is possible to extend $G(n,r)\rightarrow [G^{c}-M_{1}]^{c} \rightarrow [G^{c}-M_{1}-M_{2}]^{c} \rightarrow ....\rightarrow [G^{c}-M_{1} -M_{2}-...-M_{n-1-r}]^{c}$,
i.e., $G(n,r)\rightarrow G(n,r+1)\rightarrow ... \rightarrow G(n,n-1)$.
\end{proof}

\section{Discussion}
In this article, we have studied  several cases, where one can or can not extend $G(n,r)$ to $G(n, r+1)$ by adding edges. These cases  depend on the
  structural features of the graph, such as order, regularity, presence of specific subgraphs. 
%When  $r < n/2$ and $n$ is even,  $G(n,r)$ can always be extended to $G(n,r+1)$. 
We have attempted to find a better upper bound of $r$, for which, all $G(n,r)$ can be extended to $G(n,r+1)$. 
Moreover, we have also shown a few results on the perfect matching in odd regular graphs.

%We have shown that  if both $n$ and $r$ are even,  when $r<2(n+2)/3$ and $n \geq 52$ or when $r<n-16$ and $n<52$, $G(n,r)$ can always be extended to $G(n,r+1)$. Moreover we have also shown a special case of $G(n,r)$ where for $n$ are $r$ are even: if $n \ge 64, r<3n/4$ or $n < 64, r<n-16$, and  there exists a complete bipartite induced subgraph of $G$, then one can extend $g(n,r)$ to $G(n,r+1)$.\\
%
%We have also shown two interesting results depending on presence of specific subgraphs in a regular graph. Firstly,  if $n$ be even, $r \geq n/2$ and  $G(n,r)$ contains an induced complete bipartite subgraph with the same $n$ vertices such that the cardinality of each partition is odd, then it is impossible to extend $G(n,r)$ to $G(n,r+1)$. Secondly, if $G(n,r)$ be a regular graph such that $n$ is even, $r \geq n/2$ and $K_{n/2}$ is a subgraph of $G(n,r)$ then $G(n,r)$ can be extended to $G(n,r')$ for any $r' \leq n-1$.\\
%
%Moreover we have shown an interesting lemma about perfect matching in odd regular graphs which states that if $r>15$ be odd, $n$ be even and $n<3r+7$, then any regular graph $G(n,r)$  has a perfect matching. A corollary of this lemma claims that  a disconnected graph $G(n,r)$, where $n$ is even and $r$ $(>15 \text { and } \ge n/4)$ is odd, always
%has a complete matching.

\end{document}